\documentclass[12pt]{article}

\title{The string of diamonds is nearly tight for rumour spreading\footnote{This work started in the \emph{Random Geometric Graphs and Their
    Applications to Complex Networks} workshop held in the Banff
    International Research Station in November 2016.  The authors thank the workshop
    organizers and BIRS for making it happen. This work was completed at Microsoft Research in
    Redmond. The authors thank Microsoft for their support.}
    \footnote{A preliminary version of this paper will appear in Proceedings of the 21st International Workshop on Randomization and Computation (RANDOM'2017). This full version includes a new result, namely Theorem~\ref{thm:attainable}.}}

\author
{
  \sc{Omer Angel}\thanks{Supported by NSERC.}\\
  Department of Mathematics, University of British Columbia, \\
    \texttt{angel@math.ubc.ca} 
  \and
  \sc{Abbas Mehrabian}\thanks{Supported by an NSERC Postdoctoral Fellowship and a Simons-Berkeley Research Fellowship.  Part of this work was done while this author was at the Simons Institute for the Theory of Computing at UC Berkeley.}\\
  Department of Computer Science, University of British Columbia\\
    \texttt{abbasmehrabian@gmail.com}
  \and
  \sc{Yuval Peres}\\
 Microsoft Research\\ \texttt{peres@microsoft.com}
}

\usepackage{amsmath,amssymb,amstext,mathtools,url,tikz,amsthm}
\newcommand{\E}{\mathbb{E}}
\renewcommand{\P}{\mathbb{P}}
\renewcommand{\Pr}[1]{\mathbb{P}\left[{#1}\right]}
\newcommand{\e}[1]{\mathbb{E}\left[{#1}\right]}
\DeclareMathOperator{\erv}{Exp}
\DeclareMathOperator{\Poi}{Poi}
\renewcommand{\S}{\mathcal{S}}

\theoremstyle{plain}
\newtheorem{theorem}{Theorem}[section]
\newtheorem{lemma}[theorem]{Lemma}

\newtheorem{corollary}[theorem]{Corollary}

\theoremstyle{definition}
\newtheorem*{definition}{Definition}

\begin{document}

\maketitle

\begin{abstract}
  For a rumour spreading protocol, the spread time is defined as the first time that everyone learns the rumour.  We compare the synchronous
  push\&pull rumour spreading protocol with its asynchronous variant, and
  show that for any $n$-vertex graph and any starting vertex, the ratio
  between their expected spread times is bounded by
  $O \left({n}^{1/3}{\log^{2/3} n}\right)$.  This improves the $O(\sqrt n)$ upper
  bound of Giakkoupis, Nazari, and Woelfel (in Proceedings of ACM
  Symposium on Principles of Distributed Computing, 2016).  Our bound is
  tight up to a factor of $O(\log n)$, as illustrated by the string
  of diamonds graph.
We also show that if for a pair $\alpha,\beta$ of real numbers, there exists infinitely many graphs for which the two spread times are $n^{\alpha}$ and $n^{\beta}$ in expectation, then $0\leq\alpha \leq 1$ and $\alpha \leq \beta \leq \frac13 + \frac23 \alpha$;
and we show each such pair $\alpha,\beta$ is achievable.
\end{abstract}



\section{Introduction}

Randomized rumour spreading is an important paradigm for information dissemination in networks with numerous applications in network science, ranging from spreading of information on the web or Twitter, to diffusion of ideas an spreading of viruses in human communities.
A well studied rumour spreading protocol is the  \emph{(synchronous) {push\&pull} protocol}, introduced by Demers, Greene, Hauser, Irish, Larson, Shenker, Sturgis, Swinehart, and Terry~\cite{DGH+87} and popularized by Karp, Schindelhauer, Shenker, and V\"ocking~\cite{KSSV00}.

\begin{definition}[Synchronous push\&pull protocol]
  Suppose that one node $s$ in a network $G$ is aware of a piece of  information, the `rumour', and wants to spread it to all nodes quickly.
  The synchronous protocol proceeds in rounds; in each round $1,2,\dots$, all  vertices perform their random actions simultaneously.
  Each vertex $x$ calls a random neighbour $y$, and the two share any information they may have:
  If $x$ knows the rumour and $y$ does not, then $x$ tells $y$ the rumour (a \emph{push} operation); if $x$ does not know the rumour and $y$ knows it, $y$ tells $x$ the rumour (a \emph{pull} operation).
  Note that this is a synchronous protocol, e.g.\ a vertex that receives a rumour in a certain round cannot also send it on in the same round, even though the vertex may be involved in multiple simultaneous calls initiated by other vertices.
  The \textbf{synchronous spread time} of $G$, denoted by $S(G,s)$, is the first time that everyone knows the rumour.
  This is a discrete random variable.
\end{definition}

A point to point communication network can be modelled as an undirected graph: the nodes represent the processors and the links represent communication channels between them.
Studying rumour spreading has several applications to distributed computing in such networks, of which we mention just two (see~\cite{FPRU90} also).
The first is in broadcasting algorithms: a single processor wants to broadcast a piece of information to all other processors in the network.
The push\&pull protocol has several advantages over other protocols: it puts less load on the edges than the naive flooding protocol; it is simple and naturally distributed, since each node makes a simple local decision in each round; no knowledge of the global state or topology is needed; no internal states are maintained; it is scalable (the protocol is independent of the size of network and does not grow more complex as the network grows) and it is robust, in that the protocol tolerates random node/link failures without the need for error recovery mechanisms.

A second application comes from the maintenance of databases replicated at many sites, e.g., yellow pages, name servers, or server directories.
Updates to the database may be injected at various nodes, and these updates must propagate to all nodes in the network.
In each round, a processor communicates with a random neighbour and they share any new information, so that eventually all copies of the database converge to the same contents.
See~\cite{DGH+87} for details.

The above protocol assumes a synchronized computation and communication model, i.e.\ all nodes take action simultaneously at discrete time steps.
In many applications and certainly for modelling information diffusion in social networks, this assumption is not realistic.
In light of this, Boyd, Ghosh, Prabhakar, Shah~\cite{Boyd2006} proposed an asynchronous model with a continuous time line.
This too is a randomized distributed algorithm for spreading a rumour in a graph, defined below.
An \emph{exponential clock} with rate $\lambda$ is a clock that, once turned on, rings at times of a Poisson process with rate $\lambda$.  

\begin{definition}[Asynchronous push\&pull protocol]
  Given a graph $G$, independent exponential clocks of rate 1 are associated with the vertices of $G$, one to each vertex.
  Initially, one vertex $s$ of $G$ knows the rumour, and all clocks are turned on.
  Whenever the clock of a vertex $x$ rings, it calls a random neighbour $y$.
  If $x$ knows the rumour and $y$ does not, then $x$ tells $y$ the rumour (a push operation); if $x$ does not know the rumour and $y$ knows it, $y$ tells $x$ the rumour (a pull operation).
  The \textbf {asynchronous spread time} of $G$, denoted by $A(G,s)$, is the first time that everyone knows the rumour.
\end{definition}

Rumour spreading protocols in this model turn out to be closely related to Richardson's model for the spread of a disease~\cite{richardson,richardson_survey}.
For a single rumour, the push\&pull protocol is almost equivalent to the first passage percolation model introduced by Hammersley and Welsh~\cite{HW_FPP} with edges having independent exponential weights (see the
survey~\cite{fpp_survey}).
The difference between the push\&pull model and first passage percolation stems from the fact that in the rumour spreading models each vertex contacts one neighbour at a time, and so the rate at which $x$ pushes the rumour to $y$ is inversely proportional to the degree of $x$.
A rumour can also be pulled from $x$ to $y$.
This happens at rate determined by the degree of $y$.
On regular graphs, the asynchronous push\&pull protocol, Richardson's model, and first passage percolation are fundamentally equivalent, assuming appropriate parameters are chosen.
For general graphs, the push\&pull model is equivalent to first passage percolation with exponential edge weights that are independent, but have different means.
Hence, the degrees of vertices play a different role here than they do in Richardson's model or first passage percolation.
A collection of known bounds for the average spread times of many graph classes is given in~\cite[Table~1]{us}.

Doerr, Fouz, and Friedrich~\cite{experimental} experimentally compared the spread time in the two time models.
They state that ``Our experiments show that the asynchronous model is faster on all graph classes [considered here].''
The first general relationship between the spread times of the two variants was given in~\cite{us}, where it was proved using a coupling argument that
\[
  \frac{\e{S(G,s)}}{\e{A(G,s)}} = \widetilde{O} \left({n}^{2/3}\right).
\]
Here and below $\widetilde O$ (and $\widetilde\Omega$) allow for poly-logarithmic factors.
Building on the ideas of~\cite{us} and using more involved couplings, Giakkoupis, Nazari and Woelfel~\cite{nazari} improved this bound to $O\left({n}^{1/2}\right)$.
In this note we improve the bound to $\widetilde{O}(n^{1/3})$.
An explicit graph was given in~\cite{us} with
\[
  \frac{\e{S(G,s)}}{\e{A(G,s)}} = \widetilde{\Omega}
  \left({n}^{1/3}\right),
\]
known as the \emph{string of diamonds} (see Figure~\ref{fig:stringdiamonds}), which shows the exponent $1/3$ is optimal.

\begin{figure}
\centerline{\begin{tikzpicture}[radius=2pt]
\node at (-0.3,0){$v_0$};
\node at (2,-0.3){$v_1$};
\node at (4,-0.3){$v_2$};
\node at (6,-0.3){$v_3$};
\coordinate (v1) at (0,0) {}; \fill (v1) circle node {};
\coordinate (v2) at (1,1) {}; \fill (v2) circle node {};
\coordinate (v3) at (1,0.4) {}; \fill (v3) circle node {};
\coordinate (v4) at (1,-1) {}; \fill (v4) circle node {};
\coordinate (v5) at (1,-.4) {}; \fill (v5) circle node {};
\coordinate (v6) at (2,0) {}; \fill (v6) circle node {};
\coordinate (v8) at (3,1) {}; \fill (v8) circle node {};
\coordinate (v9) at (3,0.4) {}; \fill (v9) circle node {};
\coordinate (v10) at (3,-1) {}; \fill (v10) circle node {};
\coordinate (v12) at (3,-.4) {}; \fill (v12) circle node {};
\coordinate (v11) at (4,0) {}; \fill (v11) circle node {};
\coordinate (v18) at (4,0) {}; \fill (v18) circle node {};
\coordinate (v13) at (5,1) {}; \fill (v13) circle node {};
\coordinate (v14) at (5,0.4) {}; \fill (v14) circle node {};
\coordinate (v15) at (5,-1) {}; \fill (v15) circle node {};
\coordinate (v16) at (6,0) {}; \fill (v16) circle node {};
\coordinate (v17) at (5,-.5) {}; \fill (v17) circle node {};
\coordinate (v19) at (7,-0.8) {}; \fill (v19) circle node {};
\coordinate (v20) at (7,-0.4) {}; \fill (v20) circle node {};
\coordinate (v21) at (7,0) {}; \fill (v21) circle node {};
\coordinate (v22) at (7,0.4) {}; \fill (v22) circle node {};
\coordinate (v23) at (7,0.8) {}; \fill (v23) circle node {};
\foreach \from/\to in {v18/v13, v18/v14,v18/v15,v18/v17, v17/v16, v13/v16,v14/v16,v15/v16,v16/v19,v16/v20,v16/v21,v16/v22,v16/v23,
  v1/v2, v1/v3, v1/v4, v2/v6, v3/v6,v4/v6,v1/v5,v6/v5, v6/v8,v6/v9,v6/v10,v6/v12, v11/v8,v11/v9,v11/v10,v11/v12}
  \draw (\from) -- (\to);
\end{tikzpicture}}
\caption{The string of diamonds graph $\S_{3,4,5}$}
\label{fig:stringdiamonds}
\end{figure}
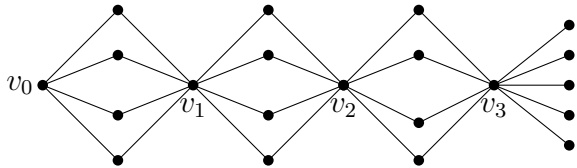

While we also a coupling argument, our argument is rather different from previous ones.
Our coupling is motivated by viewing rumour spreading as a special case of first passage percolation.
This novel approach involves carefully intertwined Poisson processes.
Our proof also yields a natural interpretation for the exponent $1/3$: using non-trivial counting arguments, we prove that the longest (discrete) distance that the rumour can traverse during a unit time interval in the asynchronous protocol is $O(n^{1/3})$ (see the proof of Lemma~\ref{P:banff}).
The string of diamonds shows that this is the best possible.

We shall make use of the following general bounds.
It is proved in~\cite{nazari} that $\e{A(G,s)}\leq \e{S(G,s)} + O(\log n)$.
Moreover, for all $G$ and $s$ (see~\cite[Theorem~1.3]{us}) we have
$$ \log n / 5 \leq \e{A(G,s)} \leq 4 n\:.$$
In this paper $n$ always denotes the number of vertices of the graph, and all logarithms are in natural base.

\section{Our results}

For an $n$-vertex graph $G$ and a starting vertex $s$, recall that $A(G,s)$ and $S(G,s)$ denote the asynchronous and synchronous spread times, respectively.  
Our main technical result is the following theorem (proved in Section~\ref{sec:proof}), which has several corollaries.

\begin{theorem}
  \label{mainthm}
  Given any $K>0$, there is a $C>0$ such that for any $(G,s)$ and any $t\geq1$ we have
  \[
    \Pr{S(G,s) > C(t + t^{2/3}n^{1/3}\log n)}
    \leq \P\Big[A(G,s)>t\Big] + Cn^{-K}.
  \]
\end{theorem}

\begin{corollary}\label{cor1}
  For any $(G,s)$, we have $\e{S(G,s)} = O \left(\e{A(G,s)}^{2/3} n^{1/3} \log n\right)$.
\end{corollary}

\begin{proof}
  Apply Theorem~\ref{mainthm} with $K=1$ and $t=3\e{A(G,s)}\leq 12n$.
  By Markov's inequality, $\Pr{S(G,s) > C(t + t^{2/3}n^{1/3}\log n)} \leq 1/3 + C/n\leq1/2$ for $n$ large enough.
  Since $t=O(n)$, this implies the median of $S(G,s)$, denoted by $M$, is $O(t^{2/3}n^{1/3}\log n)$.
  To complete the proof we need only show that $\e{S(G,s)}=O(M)$.
  Consider the protocol which is the same as synchronous push\&pull, except that, if the rumour has not spread to all vertices by time $M$, then the process reinitializes.
  Clearly the spread time for this model is larger than the spread time for the synchronous model.
  Coupling the new process with push\&pull, we obtain for any $i\in\{0,1,2,\dots\}$ that $\Pr{S(G,s) > iM} \le 2^{-i}.$
  Thus,
\[
\e{S(G,s)}= 
\sum_{i=0}^{\infty}\Pr{S(G,s)>i}\leq
\sum_{i=0}^{\infty}
M \times \Pr{S(G,s)> iM}
\leq
M
\times \sum_{i=0}^{\infty} 2^{-i}
=2M.\quad \qedhere
\]
\end{proof}

Since for all $G$ and $s$, $\e{A(G,s)}=\Omega(\log n)$, we also obtain:

\begin{corollary}\label{cor2}
  For any $(G,s)$ we have
  $$\frac{\e{S(G,s)}}{\e{A(G,s)}}= O \left({n}^{1/3}{\log^{2/3} n}\right).$$
\end{corollary}

This corollary is tight up to an $O(\log n)$ factor: consider the following construction.

\begin{definition}[$\S_{m,k,l}$]
  Let $m\geq1$, $k\ge 2$ and $l\geq 0$ be integers. 
  The graph $\S_{m,k,l}$ is built as follows.
  Start with $m+1$ vertices $v_0,v_1,\dots,v_m$.
  For each $0\le i \le m-1$,
  add $k$ edge-disjoint paths of length 2 between $v_i$ and $v_{i+1}$.
  Finally, add $l$ new vertices and join them to $v_m$ (see Figure~\ref{fig:stringdiamonds} for an example).
  The graph $\S_{m,k,l}$ has $m(k+1)+l+1$ vertices and $2km+l$ edges.
  If $l=0$, this is called a `string of diamonds' in~\cite{us}.
\end{definition}

The spread times of this graph are given by the following lemma, whose proof can be found in Section~\ref{sec:lowerbounds}.

\begin{lemma}
  \label{lem:stringofdiamondscovertimes}
  We have $\e{S(\mathcal{S}_{m,k,l},v_0)} =\Theta(m)$ and $\e{A(\mathcal{S}_{m,k,l},v_0)} = \Theta (\log n + m/\sqrt k)$.
\end{lemma}

If we let $m=\Theta(n^{1/3}(\log n)^{2/3})$ and
$k = \Theta((n/\log n)^{2/3})$
such that $km+m+1=n$, we obtain a graph $\mathcal{S}_{m,k,0}$ with
$$\frac{\e{S(\mathcal{S}_{m,k,0},v_0)}}{\e{A(\mathcal{S}_{m,k,0},v_0)}}= \Omega \left({n}/{\log n}\right)^{1/3},$$
which means Corollary~\ref{cor2} is tight up to an $O(\log n)$
factor.

It turns out that using our results and the above construction, we can get a more refined picture of what values the pair $(A(G,s),S(G,s))$ can take.
More precisely, for $\alpha, \beta$, we say the pair of exponents $(\alpha,\beta)$ is \textbf{attainable} if there exist infinitely many graphs $(G,s)$ for which $\e{A(G,s)}=\widetilde{\Theta}(n^{\alpha})$ and $\e{S(G,s)}=\widetilde{\Theta}(n^\beta)$.
One may wonder which pairs $(\alpha,\beta)$ are attainable?
The following theorem answers this question.

\begin{theorem}\label{thm:attainable}
  The pair $(\alpha,\beta)$ is attainable if and only if $0\leq\alpha \leq 1$ and $\alpha \leq \beta \leq \frac13 + \frac23 \alpha$.
\end{theorem}

\begin{proof}
  The necessity of $ 0\leq \alpha \leq \beta\leq 1$ follows from results in \cite{us} mentioned above.
  Corollary~\ref{cor1} gives $\beta \leq \frac 13 + \frac23\alpha$.

  To see that these conditions are sufficient, assume $(\alpha,\beta)$ satisfy $0\leq\alpha \leq 1$ and $\alpha \leq \beta \leq \frac13 + \frac23 \alpha$. 
  If $\beta>0$, let $m=[n^\beta/2]$, $k=[n^{2\beta-2\alpha}]$, and $l=n-1-m(k+1)$ so that $l\ge0$ for $n$ large enough.
  Lemma~\ref{lem:stringofdiamondscovertimes} gives $\e{S(\S_{m,k,l},v_0)}=\Theta(m)=\Theta(n^\beta)$ and $\e{A(\S_{m,k,l},v_0)}=\Theta(\log n + m/\sqrt k)=\Theta(\log n + n^\alpha)=\widetilde{\Theta}(n^\alpha).$
  If $\beta=0$, then $\alpha=0$; In this case the star graph on $n$ vertices  has $\e{S(G,s)}=\Theta(1)$ and $\e{A(G,s)}=\Theta(\log n)=\widetilde{\Theta}(1)$ for any vertex $s$, as required (this is because the maximum of $n$ independent exponential random variables of mean $1$ is $\log n$, see~\cite[Section~2.2]{us} for details).
\end{proof}

\section{Proof of Theorem~\ref{mainthm}}
\label{sec:proof}

In this section we fix the graph $G$ and the starting vertex $s$.
We first introduce several notations.
For any vertex $v\in G$, let $\Gamma({s,v})$ be the set of all simple paths in $G$ from $s$ to $v$.
For a path $\gamma$, let $E(\gamma)$ be its set of edges and $|\gamma| \coloneqq |E(\gamma)|$ denote its length.
Let $\deg(u)$ denote the degree of a vertex $u$.

For any ordered pair $(u,v)$ of adjacent vertices, let $Y_{u,v}$ be an exponential random variable with rate $1/\deg(u)$, so that these random variables are all independent.
In the asynchronous protocol, since each vertex $u$ calls any adjacent $v$ at a rate of $1/\deg(u)$, we can write:
\begin{equation}
  \label{asynchronous_max_min_form}
  A \coloneqq A(G,s) = \max_{v\in V}
  \min_{\gamma \in \Gamma({s,v})}
  \sum_{xy\in E(\Gamma)}
  \min \{ Y_{x,y},Y_{y,x} \}.
\end{equation}
To see this, simply interpret $Y_{x,y}$ is the time it takes after one of $x,y$ learns the rumour before $x$ calls $y$.

For any positive integer $L$, we introduce the restriction to short paths
\[
  A_L \coloneqq \max_{v\in V}
  \min_{\substack{\gamma \in \Gamma({s,v})\\ |\gamma|\leq L}}
  \sum_{xy\in E(\Gamma)}
  \min \{ Y_{x,y},Y_{y,x} \}.
\]
For any $L$ we trivially have $A_L \geq A$.
To bound $A$ from below, we have the following result giving stochastic domination ``with high probability''.

\begin{lemma}\label{P:banff}
  There exists a $C_0$ such that for any $C>C_0$, $t\geq1$ and $L \geq Ct^{2/3}n^{1/3}$ we have
  \[
    \Pr{A_L>t} \leq \Pr{A>t} + e^{-L}.
  \]
\end{lemma}

\begin{proof}
  We show that, in the asynchronous protocol, with probability $1-e^{-L}$, during the interval $[0,t]$, the rumour does not travel along any simple path of length $L$.
  This automatically implies the rumour also does not travel along any longer path either.
  We prove this by taking a union bound over all paths of length $L$.
  As there is no simple path of length $n$ or more, we may assume $L<n$.
  
  Consider a path $\gamma$ with vertices $\gamma_0,\gamma_1,\dots,\gamma_L$.
  In order for the rumour to travel along $\gamma$, it is necessary that calls are made along the edges of $\gamma$ in the order given by $\gamma$, at some sequence of times $0 \leq t_1 < \dots < t_L \leq t$.  
  Since along each edge the rumour can travel via a push or a pull, the rate of calls along an edge $xy$ is $1/\deg(x) + 1/\deg(y)$.
  Since the volume of the $L$-dimensional simplex of possible sequences $(t_i)$ is $t^L/L!$, the probability of such a sequence of calls along the path $\gamma$ is at most
  \begin{equation}\label{eq0}
    \frac{t^L}{L!} \prod_{i=1}^L \bigg( \frac{1}{\deg(\gamma_{i-1})} +
    \frac{1}{\deg(\gamma_i)} \bigg)
    \leq 
    \left(\frac{2et}{L}\right)^L Q(\gamma),
  \end{equation}
  where we define
  \[
    Q(\gamma) \coloneqq \prod_{i=1}^{|\gamma|}
    \frac{1}{\min(\deg(\gamma_{i-1}), \deg(\gamma_i))}.
  \]
  Our objective is therefore a bound for $\sum_{|\gamma|=L} Q(\gamma)$.

  For a path $\gamma$ of length $L$, consider the sequence of degrees $(\deg(\gamma_i))_{i=0}^{L}$.
  We say the sequence has a \emph{local minimum} at $i$ if $\deg(\gamma_{i-1}) > \deg(\gamma_i) \leq \deg(\gamma_{i+1})$, and a \emph{local maximum} at $i$ if $\deg(\gamma_{i-1}) \leq \deg(\gamma_i) > \deg(\gamma_{i+1})$.
  In both of these definitions we use the convention that inequalities  involving $\gamma_{-1}$ or $\gamma_{L+1}$ always hold.
  The edge set of $\gamma$ can be partitioned into \emph{segments} starting and ending at local maxima.
  For example, suppose $L=7$ and the degree sequence is
  \[
    (\deg(\gamma_0), \dots, \deg(\gamma_7)) = (\boldsymbol 5,5,7, \boldsymbol 3,4,4, \boldsymbol 2,5).
  \]
  The local minima are shown in bold.
  Then the segments are $(\boldsymbol{\pmb \gamma_0},\gamma_1,\gamma_2)$, $(\gamma_2,\boldsymbol{\pmb \gamma_3},\gamma_4,\gamma_5)$, and $(\gamma_5,\boldsymbol{\pmb\gamma_6},\gamma_7)$.
  Thus, in each segment the degrees strictly decrease to a local minimum (again, in bold), then weakly increase up to the local maximum at the end of the segment.
  (The first and last segments are special in that the local minimum could be   at the beginning and end of the segment, respectively.)
  Henceforth, we use the term \emph{segment} for a path with degrees having this property.  

  Each path gives rise to an ordered sequence of segments.
  Denote the segments of $\gamma$ by $\sigma_1,\dots,\sigma_s$, and note that $s\leq L/2+1$, since each segment (except possibly the first and the last ones) contains at least two edges.
  The next observation is that we have $Q(\gamma) = \prod Q(\sigma_i)$; that is, the $Q$ value of a path equals the product of $Q$ values of its segments (this is true for any  partition of a path into sub-paths).
  Note also that not every sequence of segments can arise in this way: each segment must start at the last vertex of the previous segment.
  Since we are interested only in simple paths, the segments are otherwise disjoint.
  Thus for a collection of segments there is at most one order in which it could arise.
  Therefore,
  \begin{equation}\label{eq1}
    \sum_{|\gamma|=L} Q(\gamma) \leq 
    \sum_{s=1}^{L/2+1}  
    \sum_{|\sigma_1|+\dots+|\sigma_s|=L}\ \frac{1}{s!} \prod_{i=1}^{s} Q(\sigma_i) \:,
  \end{equation}
  where the second sum is over ordered $s$-tuples of segments whose lengths add up to $L$, but \emph{without} the condition that they form a path (that is why we have an inequality rather than an equality).
  The $1/s!$ factor comes from the aforementioned fact that each at most one order of each $s$-tuple needs to be counted.

  We now bound the right-hand-side of~(\ref{eq1}).
  We say a segment has \emph{type} $(x,\ell^-,\ell^+)\in V(G)\times \mathbb{Z}\times\mathbb{Z}$ if the local minimum is at a vertex $x$ (called the \emph{center} of the segment), and the segment has $\ell^-$ edges before $x$ and $\ell^+$ edges after $x$.
  (The example path above had $s=3$ segments, of types $(\gamma_0,0,2)$, $(\gamma_3,1,2)$, and $(\gamma_6,1,1)$ respectively.)
  For a segment $\sigma$, let $\pi(\sigma)$ denote its type, and let $\mathcal T$ denote the set of all possible types.

  For bounding the right-hand-side of~(\ref{eq1}), we first fix $s$ and bound the number of options for the sequence $(\pi(\sigma_1),\dots,\pi(\sigma_s))$.
  There are $n^s$ choices for the centres, and at most $2^L$ choices for the lengths $\ell^\pm$ (the number of ways to write $L$ as an ordered sum of natural
  numbers).
  Thus there are at most $ 2^L n^s$ options for $(\pi(\sigma_1),\dots,\pi(\sigma_s))$.
  Enumerate these $s$-vectors of types by
  ${T}_1,\dots,{T}_m \in \mathcal{T}^s$ with $m\leq  2^L n^s$, and let $T_{j,k}$ denote the $k$th component of $T_j$, i.e.\ the type specified for $\sigma_k$ in $T_j$.
  Thus,
  \begin{align}
   \sum_{|\sigma_1|+\dots+|\sigma_s|=L}\  \prod_{i=1}^{s} Q(\sigma_i) & =
   \sum_{j=1}^m\ 
   \sum_{(\pi(\sigma_1),\dots,\pi(\sigma_s))
   = T_j} \ \prod_{i=1}^{s}  Q (\sigma_i)\notag\\\notag
   & \leq \sum_{j=1}^m \ 
   \prod_{k=1}^{s}  \left( \sum_{\pi(\sigma_k) =T_{j,k}} Q(\sigma_k)\right).
  \end{align}
  
  Next, we claim that each term in the last product, which is the sum of $Q$ values of segments of a given type can be bounded by 1.
  Fix some type $(x,\ell^-,\ell^+)$, and let $\ell=\ell^-+\ell^+$.
  The constraints on the degrees along a segment $\sigma=v_0,v_1,\cdots,v_{\ell^-},\cdots, v_{\ell}$ of this type imply $x=v_{\ell^-}$ and 
  \[
    Q(\sigma) =
    \prod_{i=1}^{\ell^-} \frac{1}{\deg(v_i)}
    \prod_{i=\ell^-}^{\ell-1} \frac{1}{\deg(v_i)}.
  \]
  If we sum this over all \emph{walks} of length $\ell^-+\ell^+$ whose $\ell^-$th vertex is $x$, but waiving the degree monotonicity constraint, then we get 1 (since the number of choices for the neighbours cancel out the degree reciprocals).
  Restricting to simple paths with piecewise monotone degrees only decreases this.
  Thus we obtain
  \[
    \sum_{|\sigma_1|+\dots+|\sigma_s|=L}\  \prod_{i=1}^{s} Q(\sigma_i)\leq m \times 1 \leq 2^L n^s.
  \]
  Plugging this back into~(\ref{eq1}) yields
  \[
    \sum_{|\gamma|=L} Q(\gamma) \leq
     \sum_{s=1}^{L/2+1} 2^L n^s/s!
    \leq \left(\frac{8en}{L}\right)^{L/2+1}.
  \]
  (We use here that $L<n$, hence each term is less than half the next and the sum is at most twice its last term.)
  
  Therefore, by~(\ref{eq0}), the probability that the rumour travels along some path of length $L$ is bounded by
  \[
    \sum_{|\gamma|=L} \left(\frac{2et}{L}\right)^L Q(\gamma)
    \leq
    \left(\frac{2et}{L}\right)^L \left(\frac{8en}{L}\right)^{L/2+1}.
    \leq
    {C_1 n} (C_2 n t^2 / L^3)^{L/2}.
  \]
  which is at most $e^{-L}$ for $L \geq C t^{2/3} n^{1/3}$, completing the proof.
\end{proof}

In (\ref{asynchronous_max_min_form}) we wrote $A(G,s)$ in a max-min form.
We would like to write $S(G,s)$ in a similar way.
To achieve this, let $q_{uv}=q_{vu}$ be the first (discrete) round at which one of $u$ or $v$ learns the rumour.
Suppose the first round \emph{strictly after} $q_{uv}$ at which $u$ calls $v$ is $F_{uv}$, and define $T_{u,v}=F_{uv}-q_{uv}$.
Note that $T_{u,v}$ is a positive integer, and is a geometric random variable: $\Pr{T_{u,v}\geq k} = (1-1/\deg(u))^{k-1}$ for any $k=1,2,\dots$.
Moreover, observe that, both $u$ and $v$ are informed by round $q_{uv}+\min \{ T_{u,v}, T_{v,u} \}$ hence, we have
\begin{equation}
  \label{def_t}
  S\coloneqq S(G,s)
  \leq \max_{v\in V}
  \min_{\gamma \in \Gamma({s,v})}
  \sum_{xy\in E(\Gamma)}
  \min \{ T_{x,y}, T_{y,x} \}.
\end{equation}

We now have a max-min expression for $S(G,s)$.
However, a major difficulty in the synchronous model is that the $\{T_{x,y}\}$ are not independent.
We will stochastically dominate them by another collection $\{X_{x,y}\}$ of random variables, which are independent.
To prove their independence, we first define the synchronous protocol in an equivalent but more convenient way.

Consider for each ordered pair $u\sim v$ a pair of exponential clocks $Z_{u,v},Z'_{u,v}$, both with rate $1/\deg(u)$.
All these clocks are independent.
Initially, the clocks $Z_{u,v}$ are turned on, and the clocks $Z'_{u,v}$ are off.
At later times we may turn off $Z_{u,v}$ and turn on $Z'_{u,v}$.
We say the clocks $Z_{u,v},Z'_{u,v}$ are \emph{located} at vertex $u$.
Continuous time at each vertex will advance separately, though there will be synchronized rounds as defined below.

For each round $1,2,\dots$, we visit the vertices one by one.
For each vertex $u$, we let all active clocks located at $u$ advance, until one of the clocks rings.
If that ring comes from clock $Z_{u,v}$ or $Z'_{u,v}$, we say that $u$ calls $v$ in that round.
Once the choice of calls at every vertex has been made, we use these to perform the push\&pull operations in a round of the protocol.
(Note that the time of the clocks is separate from the discrete rounds of
the synchronous protocol: in each vertex, a different amount of time has elapsed on the clocks.)
Having determined the spread of the rumour at the present round, whenever a vertex $u$ gets informed of the rumour, for each adjacent $v$ we turn off the clocks $Z_{u,v}$ and $Z_{v,u}$, and turn on $Z'_{u,v}$ and $Z'_{v,u}$.
(If $v$ was already informed, these status changes will have already taken place at an earlier round.)
Observe that, because of memorylessness of the exponential distribution, and since all clocks at $u$ have the same rate, this process generates a random sequence of independent uniform neighbours, so it is equivalent to the synchronous protocol.

Now let us see what are the random variables $T_{u,v}$ in this setup.
For each ordered pair $u,v$, observe that the combined collection of ringing times of clocks $Z_{u,v},Z'_{u,v}$ forms a Poisson process $P_{u,v}$ with rate $1/\deg(u)$.
(It does not matter that several initial rings come from $Z$ and subsequent rings from $Z'$.)
Let
\[
  P_u \coloneqq \bigcup_{v\sim u} P_{u,v},
\]
and note that $P_u$ is a Poisson process with rate 1.

For a pair $u,v$, suppose the $q_{uv}$th point in $P_u$ is at $\alpha$, and suppose the first point of $P_{u,v}$ strictly larger than $\alpha$ is at $\beta$.
Then, $T_{u,v}$ is precisely the number of points of $P_u$ in the interval $(\alpha,\beta]$.
Define $X_{u,v} = \beta - \alpha$. 
By construction, $X_{u,v}$ is the first time that clock $Z'_{u,v}$ rung from the time it was turned on, hence it is exponential with rate $1/\deg(u)$.
Since the clocks are independent, the random variables $X_{u,v}$ are also independent.
The times at which the $Z'$ clocks are turned on depend on other clocks in a non-trivial manner, but does not affect the $X_{u,v}$ variables.
Thus we have proven:

\begin{lemma}\label{declaration_of_independence}
  The random variables $\{X_{u,v}\}$ defined above are mutually independent.
\end{lemma}

On the other hand, we can use these to control the $T_{x,y}$:

\begin{lemma}
  \label{lem:coupling}
  For every $K$ and large enough $C \geq C_0(K)$, with probability at least $1 - n^{-K}$, for all adjacent pairs $u,v$ we have $T_{u,v} \leq C \log n + C X_{u,v}$.
\end{lemma}

\begin{proof}
  We show that for any adjacent pair $x,y$, we have $\P(T_{u,v} > C \log n + C X_{u,v}) \leq n^{-K-2}$, and then apply the union bound over all edges.
  
  Observe that, conditioned on $X_{u,v}=t$, the random variable $T_{u,v}-1$ is Poisson with rate $t \times (\deg(u)-1)/\deg(u)\leq t$.
  Indeed, this is the number of rings over time $t$ of the $\deg(u)-1$ active clocks on edges $(u,w)$ with $w\neq v$.
  Let $\Poi(t)$ denote a Poisson random variable with mean $t>0$.
  For $m\geq et$, we have $\P(\Poi(t) = m) \leq e^{-1} \P(\Poi(t)=m-1)$, hence $\P(\Poi(t) > et+m) \leq e^{-m}$.
  This gives
  \begin{align*}
    \Pr{T_{u,v}-1 > (K+2)\log n + e X_{u,v} | X_{u,v}=t} 
    & \leq \Pr{\Poi(t) > (K+2)\log n + e t} \\
    & \leq n^{-K-2}.
  \end{align*}
  The claim follows with $C = \max(e,K+2)$.
\end{proof}

Our main result now follows easily from our lemmas.

\begin{proof}[Proof of Theorem~\ref{mainthm}]
  Given $K$, pick $C$ sufficiently large so that Lemmas~\ref{P:banff} and~\ref{lem:coupling} hold.
  Fix $t\geq1$ and let $L=Ct^{2/3}n^{1/3}$.
  We have
\begin{align*}
&\Pr{S > C t + C L \log n}
\\
&  \leq \P\Bigg[
\Bigg(\max_{v\in V}
\min_{{\gamma \in \Gamma({s,v})}}
\sum_{xy\in E(\gamma)}
\min \{ T_{x,y}, T_{y,x} \}\Bigg)  > Ct  + C L \log n \Bigg]\\
&\leq \P\Bigg[
\Bigg(\max_{v\in V}
\min_{\substack{\gamma \in \Gamma({s,v})\\{{|\gamma|\leq L} }}}
\sum_{xy\in E(\gamma)}
\min \{ T_{x,y}, T_{y,x} \}\Bigg)  > Ct  + C L \log n \Bigg]\\
& \leq 
\P \Bigg[
\Bigg(\max_{v\in V}
\min_{\substack{\gamma \in \Gamma({s,v})\\{{|\gamma|\leq L} }}}
\sum_{xy\in E(\gamma)}
C \log n + C \min\{X_{x,y},X_{y,x}\} \Bigg)  > Ct  + C L \log n \Bigg]
+ n^{-K} \\
& \leq \P \Bigg[
\Bigg(\max_{v\in V}
\min_{\substack{\gamma \in \Gamma({s,v})\\{{|\gamma|\leq L} }}}
\sum_{xy\in E(\gamma)}
 C \min \{X_{x,y}, X_{y,x}\}
\Bigg)  > Ct  \Bigg] + n^{-K}\\
& = \Pr{ A_L  > t } + n^{-K}\\
& \leq \Pr{A  > t } + n^{-K} + e^{-Cn^{1/3}}.
  \end{align*}
  Here, the first inequality is copied from~\eqref{def_t}.
  The second inequality is because restricting the feasible region of a minimization problem can only increase its optimal value.
  The third inequality follows from Lemma~\ref{lem:coupling}.
  The fourth inequality is straightforward.
  The equality follows from the definition of $A_L$ and noting that $\{X_{x,y}\}$ have the same joint distribution as $\{Y_{x,y}\}$,
  and the last inequality follows from Lemma~\ref{P:banff}.
  This completes the proof of Theorem~\ref{mainthm}.
\end{proof}

\section{Proof of Lemma~\ref{lem:stringofdiamondscovertimes}}
\label{sec:lowerbounds}

In this section we show that $2m \leq \e{S(\mathcal{S}_{m,k,l},v_0)} \leq 4m+1$ and $\e{A(\mathcal{S}_{m,k,l},v_0)} = \Theta (\log n + m/\sqrt k)$.
Fix $m\geq1$, $k\ge 1$ and $l\geq 0$, and let $G = \S_{m,k,l}$.
Recall that $v_0,\dots,v_m$ are the vertices connecting the diamonds in $\S_{m,k,l}$

Since the graph distance between $v_0$ and $v_m$ is $2m$, we have $S(G,v_0) \geq 2m$ deterministically.
Fix $0\leq i\leq m-1$ and suppose that at some time $v_i$ is informed and $v_{i+1}$ is uninformed.
We claim that the expected time to inform $v_{i+1}$ is at most 4.
Let $u$ be some common neighbour of $v_i$ and $v_{i+1}$.
It takes 2 rounds in expectation for $u$ to pull the rumour from $v_i$, and another 2 rounds for it to push the rumour to $v_{i+1}$, so the claim follows.
Once all the $v_i$ are informed, every other vertex will be informed in the next round.
Therefore, $\e{S(G,v_0)} \leq 4 m + 1$.

Next we show $\e{A(G,v_0)} = O (\log n + m/\sqrt k)$.
Let $Y_i$ denote the communication time between $v_i$ and $v_{i+1}$ (the first time that $v_{i+1}$ learns the rumour, assuming initially only $v_i$ knows the rumour).
Between $v_i$ and $v_{i+1}$ there are $k$ disjoint paths of length 2, so $Y_i$ is stochastically dominated by $Z:=\min\{Z_1,\dots,Z_k\}$, where the $Z_i$ are independent random variables equal in distribution to the sum of two independent exponential random variables with rate $1/2$.
(The difference between $Y$ and $Z$ stems from calls initiated at $v_i,v_{i+1}$.)
Since each $Z_i$ has density $(t/4) e^{-t/2}$ on $\mathbb{R}_+$, we have 
$\P[Z>t] = \left(1+\frac{t}{2}\right)^k e^{-kt/2}$.
The change of variable $u=k(t/2+1)$ gives
\[
  \E[Z] = \int_0^\infty \P[Z>t] dt
  = \frac{2e^k}{k^{k+1}} \int_k^\infty u^k e^{-u} du.
\]
The integral from $0$ to $\infty$ is $k!$, so
\[
  \E[Z] \leq \frac{2e^k k!}{k^{k+1}} = O(1/\sqrt{k}).
\]
Hence, the expected time for all the $v_i$'s to learn the rumour is at most $O(mk^{-1/2})$. 
After this has happened, any other vertex pulls the rumour in $\erv(1)$ time.
The expected value of the maximum of at most $n$ independent $\erv(1)$ variables is the harmonic sum $H_n \leq 1+\log n$, so $\e{A(G,v_0)}=O(\log n+mk^{-1/2})$.

Finally, we show $\e{A(G,v_0)}=\Omega(\log n+mk^{-1/2})$.
The bound $\e{A(G,v_0)}=\Omega(\log n)$ holds for any $n$-vertex graph $G$ (see~\cite[Theorem~1.3]{us}), so we need only show that $\e{A(G,v_0)}=\Omega(mk^{-1/2})$.
In fact, since each of the intermediate $v_i$ is a cut vertex, we need only show that $\e{Y_i}=\Omega(k^{-1/2})$ for each $i$.

Suppose that at time $s$ only $v_i$ is informed.
For any $t>0$, if $v_{i+1}$ is informed by time $s+t$, then during the time interval $[s,s+t]$, either the clock of $v_i$ has rung at least once, or the clock of $v_{i+1}$ has rung at least once, or the clock of one of their $k$ common neighbours has rung at least twice.
Since the ringing times at each vertex are a Poisson process, we find
\begin{align*}
\Pr{Y_i\leq t} & \leq 2 (1-e^{-t}) + k (1-e^{-t}-te^{-t}) \leq 2t+kt^2/2.
\end{align*}
Hence, with $t=1/3\sqrt{k} \leq 1/3$,
\[
  \E[Y_i] \geq \frac{1}{3\sqrt k} \Pr{Y_i\geq \frac{1}{3\sqrt k}}
  \geq 
  \frac{1}{3\sqrt k} (1-2/3-1/18) = \Omega(1/\sqrt k) \:,
\]
completing the proof of Lemma~\ref{lem:stringofdiamondscovertimes}.

\subparagraph*{Acknowledgments.}

We would like to thank an anonymous referee of the RANDOM'2017 workshop for pointing out a mistake in the statement of Corollary~\ref{cor2} in an earlier version.

\end{document}